\documentclass[11pt]{amsart}
\usepackage{geometry}                
\usepackage{graphicx}
\usepackage{amssymb}
\usepackage{epstopdf}
\newtheorem{theorem}{Theorem}[section]
\newtheorem{lemma}[theorem]{Lemma}

\theoremstyle{definition}
\newtheorem{definition}[theorem]{Definition}

\theoremstyle{remark}
\newtheorem{remark}[theorem]{Remark}

\numberwithin{equation}{section}

\newcommand{\p}{\partial}
\newcommand{\wt}{\widetilde}

\newcommand{\BFR}{\mathbf{R}}
\newcommand{\BFN}{\mathbf{N}}
\newcommand{\BFC}{\mathbf{C}}

\newcommand{\ep}{\epsilon}

\newcommand{\be}{\begin{equation} }
	\newcommand{\ee}{\end{equation}}
\newcommand{\bse}{\begin{subequations}}
	\newcommand{\ese}{\end{subequations}}

\begin{document}

\title{On the rapidly rotating vorticity in the unit disk}

\author{Yuchen Wang$^{1,2}$}

\address{${}^1$ School of Mathematics and Statistics, Central China Normal University\\ Wuhan 430079, Hubei, China}
\address{${}^2$ School of Mathematical Science, Tianjin Normal University\\ Tianjin, 300387,  China}
\email{wangyuchen@mail.nankai.edu.cn}
\thanks{Partially supported by the National Science Foundation of China No.11831009 and Hubei Province Science and Technology Innovational Funding. }

\maketitle

\begin{abstract}
In this paper, we obtain uniformly rotating vorticity with sufficiently large angular velocity in the unit disk.  The solution consists of either a small nearly-ellipse vortex patch which is highly concentrated near the origin or a $2+1$ configuration in which the another two vortical components are very close to the boundary of fluid domain. The total vorticity are prescribed and the sizes of vortical domains are not necessarily small in the same order. 

The construction, which is based on a perturbation argument, exhibits a subtle multiscale phenomenon in this highly rotating fluid. \\

{\bf Keywords}: Rotating vortex patch; Kirchhoff ellipse patch; Concentrated vorticity; Conformal mapping; Robin function
\end{abstract}

\section{Introduction} \label{S:Intro}

Consider the 2-D  incompressible Euler equation in the unit disk $ D:=\{ x \in \BFR^2 | \;|x| \leq 1 \}$ with the slip boundary condition
\be \label{E:Euler-E}
\begin{cases}
	 \p_t u + (u \cdot \nabla) u = - \nabla p,   & \quad x \in D, \\
	 \nabla \cdot u = 0, & \\
	u \cdot \BFN = 0, & \quad x \in S^1,
\end{cases}
\ee
where $u=(u^1,u^2)^t$ is the velocity of fluid, $p$ is the pressure and $\mathbf{N}$ denotes the unit outward normal to $\p D$. Under mild smooth condition, it is equivalent to work with the vorticity formulation of the 2-D incompressible Euler equation 
\be \label{E:Euler-V}
 \p_t \omega + u \cdot \nabla \omega = 0
\ee
with the vorticity field given by $\omega := \p_1 u^2 -\p_2 u^1$. It is a well-known closed PDE system for $\omega$ which is globally well-posed for classical solutions. The global well-posedness for weak solution $\omega \in L^\infty(D)$ was obtained by Yudovich in \cite{Yud1963}.

Special interest is addressed on the vortex patch given as $\omega = \lambda \mathbf{I}_K$ where $\lambda$ is non-zero constant, $K$ is bounded domain and $\mathbf{I}$ is the characteristic function. It is clearly that, due to Yudovich's result, vortex patch defines a unique 2-D Euler flow globally in the time. Since the vorticity is only transported by the Euler flow due to \eqref{E:Euler-V}, evolution of vortex patch is indeed boundary dynamics of vortical domain given by
\be \label{E:Dyn-VP}
\p_t z(s,t) \cdot \mathbf{n} = u\left(z(s,t),t\right) \cdot \mathbf{n}, \; \quad s \in S^1, 
\ee
where  $z(s,t)$ is the arclength parameterization of $\p K(t)$ and $\mathbf{n}$ denotes the unit outward normal to $\p K(t)$. This equation is highly nonlinear and nonlocal hence its dynamical behaviors would be very rich. On the other hand, the global regularity of vortex patch was proved by Chemin in \cite{Che1993} provided the initial data $\p K(0) \in C^{k,\alpha}(S^1)$, $k \in \mathbb{N}, 0<\alpha<1$, namely the scenario of finite-time blow-up is ruled out. See a different proof given by Bertozzi and Constantin in \cite{BC94}.\\ 

The central issue in patch dynamics is to investigate the long-time dynamics. As the first step, it is essential to well understand  structure of steady states for \eqref{E:Dyn-VP}. 
Because of the $SO(2)$ invariance, we shall focus on the uniformly rotating vorticity or relative equilibrium rather than absolute equilibrium. 
\begin{definition}
Let $\omega_0$ be the initial data. $\omega$ is a uniformly rotating vorticity if the following condition
\be
\omega(t,x) = \omega_0(e^{-it \Omega}x), \quad x \in D, 
\ee
holds for some angular velocity  $\Omega \in \BFR$.
\end{definition}
It is straightforward that vortex patch $\omega = \lambda \mathbf{I}_K$ is uniformly rotating if
\be \label{E:Steady-V}
-\Omega z^\perp \cdot  \mathbf{n} =\p_t z \cdot \mathbf{n} = u \cdot \mathbf{n}, \quad z \in \p K(t), \quad x^\perp := (x_2,-x_1),
\ee
which is completely determined by the initial data.  It is a trivial conclusion that radial solutions including circular and annular patches, are uniformly rotating vortex patch for any $\Omega \in \BFR$. By carefully studying the linearization on trivial patches, lots of $m$-fold non-trivial rotating vortex patch on the whole plane were obtained due to local bifurcation argument \cite{CR1971}, see for example \cite{Burbea1982,Hmidi2013,Hoz2016a,Burbea1982} and references therein. In the unit disk, the local bifurcation curve emanating from circular/annular vortex patches was obtained by de la Hoz, Hassainia, Hmidi and Mateu in \cite{Hoz2016}. 

On the other hand there is also another explicit non-trivial rotating vortex patch on the whole plane called  the Kirchhoff ellipse vortex patch 
\[
\omega = \lambda \mathbf{I}_{E_{a,b}},\text{ where } E_{a,b}:=\left\{(x,y) \in \BFR^2 \big| \frac{x^2}{a^2} + \frac{y^2}{b^2} \leq 1 \right\}.
\]
It rotates with the angular velocity $0<\Omega = \frac{\lambda ab}{(a+b)^2} \leq \frac{\lambda}{4}$. It is well-known that the Kirchhoff ellipse vortex patch is linear stable for $b/a>\frac{1}{3}$ and unstable mode occurs and grows expontentially if $b/a<\frac{1}{3}$. Hmidi and Mateu \cite{Hmidi2016} and  Castro, C\'{o}rdoba and G\'{o}mez-Serrano in \cite{Castro2016} obtained a family of one-fold uniformly rotating vortex patch bifurcated from the Kirchhoff ellipse vortex patch at the critical ratio $b/a =\frac{1}{3}$. Note that the Kirchhoff ellipse patch is on the $2$-fild bifurcation emanating from disk. These results affirm that a secondary bifurcation occurs.

As the fluid domain is less symmetrical, on the other hand, there is no known explicit trivial solution of \eqref{E:Euler-V} so far. To construct steady vorticity, an alterhative approach called {\it desingularization } is widely implemented. The essential ingredient is the celebrated vortex point model proposed by Kirchhoff which is used to descirbe evolution of concentrated vorticity approximately. A mathematical rigorous proof concerned with the validity of vortex points model can be found in the monograph \cite{MP1994}. Note that vortex point system is a finite-dimensional approximating of the 2-D Euler equation. A key observation is that the steady vorticity would persist near the configuration of steady vortex points. In the last decades, several approaches are developed, for example singular perturbation argument \cite{Wan1988,LWZ2019}, singular perturbed elliptic equations \cite{Cao2014,Cao2015,Smets2010} and variational approach \cite{Tur83,CW19,Cao2021,Wan2021}. Vortical domains constructed in these papers are perturbation of small disks. \\

In this paper, we obtained a different type rotating vortex patch. They are highly concentrated near the orgin or the boundary of fluid domain with non-vanishing total vorticity in each vortical component. They manifest a nearly ellipse vortex patch located on the origin driving the fluid for a rapidly rotation. More precisely, let us consider simply-connected vortex patch at first. The first result is read as follow
\begin{theorem} \label{T:main-1}
Suppose $0 < Q <\frac{1}{2}$ and $s>\frac{3}{2}$. There exists $\ep_0>0$ such that for any $0<\ep < \ep_0$, there exists a family of one-fold uniformly rotating vortex patch $\omega = \frac{1}{ \pi \ep^2}\mathbf{I}_K$ with angular velocity $\Omega = \frac{1}{4\pi \ep^2}(1-Q^2)$ with the prescribed total vorticity $\int_D \omega = 1-Q^2$. The vortex patch is concentrated near the origin and  $\p \left(\ep^{-1}  K\right)$ is a $H^s$ closed curve which is an $O(\ep^2)$ perturbation of the elliptic curve
\be
\mathbf{S}_Q:=\{ (x,y) \in \BFR^2 \;\;\big| \frac{x^2}{(1+Q)^2} + \frac{y^2}{(1-Q)^2} = 1 \}.
\ee
in the $H^s(S^1)$ topology.
\end{theorem}
\begin{remark}
It would be an interesting issue whether this nearly ellipse vortex patch locates on the bifurcation curve emanating from a small Rankine vortex with prescribed total vorticity. We believe that the normalized boundary $\p \left(\ep^{-1} K \right)$ would be transcendental rather than algebraic.   
\end{remark}
\begin{remark}
The linear/nonlinear (in)stabiltiy of rotating vortex patch would recovered from the case on whole plane \cite{Wan1986, Guo2004} with slightly modifications. A detailed analysis on the local dynamics of the rotating vortex patch would be given in a forthcoming paper.   
\end{remark}
If another vortical domain is involved, the situation is more subtle. On the one hand, it is hard to expert that a vortical domain largely deviated from disk is steady if its location is far from the origin. On the other hand, since a small Kirchhoff ellipse vortex patch rotates very fast due to the total vorticity is prescribed, the location of another vortical domain would be very close to the boundary of fluid domain.    

\begin{theorem} \label{T:main-2}
Suppose $0<Q<\frac{1}{2}$, $s>\frac{3}{2}$ and $\mu>0$. There exists $\ep_0>0$ such that for $\vec{r}=(r_0, r_1,r_2) \in \BFR_+^3$ with $|\vec{r}| < \ep_0$ in which $r_1,r_2 < \frac{r_0^2}{2}$, there exists a family of uniformly rotating vorticity 
\[
\omega = \frac{1}{\pi r_0^2} \mathbf{I}_{D_0} + \omega_1  + \omega_2 
\]
with angular velocity $\Omega = \frac{1-Q^2}{4\pi r_0^2} + \frac{\mu}{4 \pi }$, where $|D_0| = \pi r_0^2(1-Q^2)$ and $\omega_j$ satisfying \eqref{E:Multi-con} can be piecewisely constant vortex patch or Lipschitz continuous vorticity. Moreover, $\p(r_0^{-1}  D_0)$ is an $O(r_0^2)$ perturbation of the elliptic curve $\mathbf{S}_Q$ and each $r_j^{-1} D_j$ is an $O(|r_j|r_0)$ perturbation of disk $B_1$ in the $H^s$-topology for $j=1,2$. The distence between the vortical domain $D_j$ and boundary of the fluid domain satisfies 
\[
d\left(D_j,\p D \right) = \frac{2 \mu}{1-Q^2} r_0^2 + O(r_0^4), \quad j=1,2. 
\]
\end{theorem}

\begin{remark}
The vortex components $\omega_j$ are neither symmetric nor small as the same order, for example $\omega_1 = \frac{\mu }{\pi r_j^2} \mathbf{I}_{D_1}$ is a vortex patch and $\omega_2 = \omega_2(x) \mathbf{I}_{D_2} \in C^{0,1}(D)$ with $\int_{D_2} \omega_2 = \mu$. The presence of extra vortical components slightly accelerates the angular velocity by an order $O(1)$ term. The condition $\max\{r_1,r_2\} <\frac{r_0^2}{2}$ is to avoid the collision of vortical domain and the boundary. 
\end{remark}

This paper is organized as follow. In Section \ref{S:Pre}, we introduce necessary backgrounds, including steady states and conformal mapping parametrization. In Section \ref{S:Sin-patch}, we consider the rapidly rotating simply-connected vortex patch and prove 
Theorem \ref{T:main-1}. Theorem \ref{T:main-2} for uniformly rotating vorticity with  multiple vortical components is proved in Section \ref{S:Multi}.

\section{Preliminaries} \label{S:Pre}

 The incompressible condition as well as the simply-connectedness of the fluid domain guarantee that there exists a unique globally-defined stream function 
\be \label{E:Vor-Stream}
- \Delta \Psi = \omega \; \text{ in } D, \quad \Psi = 0  \;  \text{ on } \p D,
\ee
such that $u = \nabla^\perp \Psi$. Therefore the stream function is explicitly represented by
\be
\Psi(x,t) = \int_{D} G_D(x,y) \omega(y,t) d \mu(y),
\ee
as well as the velocity
\be \label{E:Velo-E} 
u(x,t) = \nabla^\perp \Psi= \int_{D} \nabla_x^\perp  G_D(x,y) \omega(y,t) d \mu(y),
\ee
where $d \mu$ denotes the Euclidean Lebesgue measure on $\BFR^2$ and  
\be \label{E:Green-f}
G_D(x,y) = - \frac{1}{2 \pi} \log|x-y| + \frac{1}{2 \pi} \log \left||y| \left(x - \frac{y}{|y|^2} \right) \right|, \; x,y \in D,
\ee
denotes the Dirichlet Green's function for the unit disk.

For vortex patch $\omega=\lambda \mathbf{I}_{K}$, the velocity would be represented in the contour integral formula. Indeed, consider it as a complex-valued function $u = u^1(x_1,x_2) + i u^2(x_1,x_2)$ for complex variable $z =x_1 + i x_2 \in \mathbb{C}$, by the Stokes' theorem, one has  
\be \label{E:Contour-F}
\begin{split}
	\overline{u(z,t)} & = - \frac{i \lambda}{2 \pi} \int_{K} \frac{1}{z - \xi } d \mu_\xi  -  \frac{i\lambda}{2 \pi } \int_{K} \frac{\overline{\xi}}{1 - z \overline{\xi}} d \mu_\xi 
	 \\
	 & =\frac{\lambda}{4 \pi} \int_{K} \frac{1}{\xi-z} d \overline{\xi} \wedge d \xi + \frac{\lambda}{4 \pi} \int_{K} \frac{\overline{\xi}}{1 - z \overline{\xi}} d \xi \wedge d \overline{\xi} \\
	& = \frac{\lambda}{4 \pi} \int_{\p K} \frac{\overline{\xi}-\overline{z}}{\xi - z} d \xi  - \frac{\lambda}{4 \pi} \overline{ \int_{K} \p_{\overline{\xi}} \left( \frac{|\xi|^2}{1 - \overline{z} \xi} d \xi \right)} \\
	& = \frac{\lambda}{4 \pi } \int_{\p K} \frac{\overline{\xi}-\overline{z}}{\xi - z} d \xi- \frac{\lambda}{4 \pi} \int_{\p K} \frac{|\xi|^2}{1 - \overline{z} \xi} d \xi.
\end{split}
\ee
Recall the inner product on complex domain $a \cdot b = {\rm Re} (\overline{a} b) = {\rm Im} (i \overline{a} b)$. \eqref{E:Steady-V} imples $\omega = \lambda \mathbf{I}_{K}$ is a uniformly rotating vortex patch if 
\be \label{E:Rotating-VP}
{\rm Im}(i \overline{u}  \mathbf{n})  = {\rm Im}( \Omega \overline{z}  \mathbf{n}),
\ee
where $\mathbf{n}$ denotes the unit outward normal to $\p K$. Inserting the formula \eqref{E:Contour-F} into \eqref{E:Rotating-VP}, it is equivalent that $\p K$ is zero point of the nonlinear functional
\begin{gather} \label{E:Steady-VP-eq}
	{\rm Im} \left( (2 \Omega \overline{z} + I(z)) \mathbf{n} \right) = 0, \; \forall z \in \p K,  
\end{gather}
where
\[
I(z) = \frac{\lambda}{2 \pi i } \int_{\p K} \frac{\overline{\xi-z}}{\xi -z} d \xi - \frac{\lambda}{2 \pi i } \int_{\p K} \frac{|\xi|^2}{1 - \overline{z} \xi} d \xi.
\]
To formulate the problem in an appropriate analytical framework, we follow the conformal mapping parameterization for unknown domain. Suppose that $\ep$ is a small prescribed parameter,  $\ep^{-1}K$ is a simply-connected bounded domain with prescribed measure $|K| = \ep^2 \pi(1 -Q^2)$ which is not far from the ellipse
\[
\mathbf{E}_Q:=\{(x,y) \in \BFR^2 \big| \frac{x^2}{(1+Q)^2} + \frac{y^2}{(1-Q)^2} \leq 1 \},
\]  
in the sense of boundary perturbation in $H^s$ topology. By the Riemannian mapping theorem, there exists a unique conformal mapping between the exterior domains   
\be \label{E:Ellipse-1}
\Phi(z)  = A \ep \left(z + \frac{Q}{z} + \sum_{ n \geq 2} \frac{B_n}{z^n} \right) :\; \BFR^2 \setminus D \rightarrow \BFR^2 \setminus K,\quad \Phi(\infty)=\infty, \;A>0, 
\ee
and $B_n \in \BFR$ due to the $S^1$-invariance. We note that $\Phi(z)$ is uniquely determined by the real part of the trace on the unit disk
\be \label{E:Ellipse-2}
\phi(\theta) = \sum_{n \geq 2} B_n \cos n \theta \in \dot{H}^s(S^1),\;\; s>\frac{3}{2}.
\ee
up to  a Hilbert transform and  
\be \label{E:A}
A = \frac{(1-Q^2)}{(1-Q^2) - \left(\sum_{n \geq 2} n B_n^2\right)} = \frac{(1-Q^2)}{(1-Q^2) - \frac{1}{2\pi}\| \phi \|^2_{\dot{H}^{\frac{1}{2}}(S^1)}}.
\ee
It follows from the fact that the area enclosed by the closed curve 
$$\p K=(x_1(\theta),x_2(\theta)) = \left({\rm Re} \Phi(e^{i\theta}), {\rm Im} \Phi(e^{i\theta}) \right), \; \theta \in S^1$$ 
is given by 
\[
\pi(1-Q^2) = |K|= \frac{1}{2} \int_{S^1} \left(x_1(\theta) x_2^\prime(\theta) -x_1^\prime(\theta) x_2(\theta) \right) d \theta.
\]  
For brevity,  $\phi(z) = \sum_{n \geq 2} \frac{B_n}{z^n}$ is also reserved for complex function unless otherwise stated. \\

If more vortical components are involved, for example,  
\[
\omega = \omega_0 \mathbf{I}_{D_0} + \omega_1(x) \mathbf{I}_{D_1} + \omega_2(x) \mathbf{I}_{D_2},
\] 
in which $\omega_0$ is a non-zero constant and each $\omega_j \in C^\infty(D_j) \cap C^{0,1}(D)$, $j=1,2$, the situation is more subtle. Firstly, let us consider the relative stream function $\psi = \Psi - \frac{\Omega}{2}|x|^2$. Similar with Lemma 2.5 in \cite{LWZ2019}, we have
\begin{lemma} \label{L:Steady-vor}
	Suppose $D_1,D_2$ are mutually disjoint and have $H^s$ boundary, $s>\frac{3}{2}$. For $j=1,2$, suppose $-\Delta \psi = f_j(\psi)$ on $D_j$ for some $f_j \in C^1$ and $\psi$ is also constant along $\p D_0$. Then $\omega$ is a uniformly rotating vorticity.
\end{lemma}   
We would like to conclude the existence of uniformly rotating vorticity into a nonlinear functional problem. Since both location and shape are unknown a prior for nearly disk vortical domains, we consider the following conformal mapping parameterization.
\begin{lemma}[Lemma 2.6 in \cite{LWZ2019}] \label{L:CM}
	Let $k\ge 2$ be an integer and $f\in C^k(\overline{D}, \BFC)$  be a conformal mapping. 
	such that  $f|_{S^1} \in H^s (S^1, \BFR)$, $s >k+\frac 12$. 
	and $U = f(D)$. 
	Then there exist $\alpha \in S^1$ and $c \in \BFC \cap D$ such that $\p_z f_1(c, \alpha, 0)>0$ and $\p_z^k f_1(c, \alpha, 0) =0$, where $f_1 (c, \alpha, z) = f \big(\alpha \frac {z+c}{1+\bar c z} \big)$. 
	Moreover, if 
	$s>k+\frac 32$ and 
	a conformal mapping $f_0 \in C^1 (\overline{D})$ satisfies 
	\be \label{E:CM-A}
	\p_z^k f_0 (0)=0, \quad |\p_z^{k+1} f_0 ( 0)|  \ne k(k-1)|\p_z^{k-1} f_0 ( 0)|
	\ee
	then such $(c, \alpha)$ are unique and $C^1$ with respect to $f$ near $f_0$ in $C^1 (\overline{D})$. 
\end{lemma}
Let $k=2$. Each nearly disk vortical domain satisfying 
\[
|D_j| = \pi r_j^2, {\rm dist}(D_j) \leq 2.5 r_j
\]
would be parameterized by a local unique conformal mapping
\be 
\Gamma(z) = x + a_1r\left( z + \wt \Gamma(z) \right), \quad \wt \Gamma(z) = \sum_{n \geq 3} (A_n - iB_n)z^n, 
\ee
in which $x$ specifies the location, $\wt \Gamma(z)$ or equivalently $\beta(\theta) = \wt \Gamma(z) \big|_{S^1}$ describes the shape of domain and $a_1$ depends on $\beta$ such that the area is $\pi r^2$. See more details in Section \ref{S:Multi}.

\section{Rapidly rotating simply-connected vortex patch} \label{S:Sin-patch}

In the beginning, we focus on the simply-connected vortex patch. Suppose $0<Q<\frac{1}{2}$ is a prescribed parameter and $0<\ep \ll 1-Q$ is a small parameter. Due to \eqref{E:Rotating-VP}, we consider the nonlinear functional  
\be \label{E:Non-patch}
\mathcal{F}(\ep \phi,\ep)(\theta) : = {\rm Im} \left((2\Omega \overline{\Phi(z)} + I(\Phi(z))) \Phi^\prime(z)z\right) : X^s \rightarrow Y^{s-1} 
\ee 
with $\Omega = \frac{1-Q^2}{4 \pi \ep^2}$ where 
\begin{align*}
X^s:= \{ f \in \dot{H}^s(S^1) \;\; \big| \; f(\theta) = \sum_{n \geq 2} a_n \cos n \theta,\; a_n \in \BFR \}, \\
Y^s:= \{ f \in \dot{H}^s(S^1) \;\; \big| \; f(\theta) = \sum_{n \geq 1} b_n \sin n \theta, \; b_n \in \BFR \},
\end{align*}
due to the $S^1$-invariance. Motiviated by the Kirchhoff ellipse vortex patch, we expert to
 obtain uniformly rotating vortex patch $\omega = \frac{1}{\pi \ep^2} \mathbf{I}_K$ with prescirbed area $|K| =\pi \ep^2(1-Q^2)$ in which $\ep^{-1}K$ is not far from $\mathbf{E}_Q$.
Let 
\be
\Phi(z) = A \ep \left(z+ \frac{Q}{z} + \ep \sum_{n \geq 2} \frac{B_n}{z^n}\right),  \phi(z) = \sum_{n \geq 2} \frac{B_n}{z^n}.
\ee
The nonlinear functional \eqref{E:Non-patch} is expressed as  
\be \label{E:non-linear}
\begin{split}
	& \mathcal{F}(\ep \phi,\ep)  = \frac{A^2(1-Q^2)}{2 \pi \ep}   {\rm Im} \left( \overline{z + \frac{Q}{z} + \ep \phi(z)} \cdot ( z - \frac{Q}{z} + \ep \phi^\prime(z)z) \right)  \\
	& + \frac{A^2}{2 \pi^2 \ep} {\rm Im} \left(-i \int_{S^1} \frac{ \overline{ \xi + \frac{Q}{\xi} - z - \frac{Q}{z}} + \ep (\overline{\phi(\xi)} - \overline{\phi(z)})}{ \xi + \frac{Q}{\xi} - z - \frac{Q}{z} + \ep (\phi(\xi) - \phi(z)) } (1 - \frac{Q}{\xi^2} + \ep \phi^\prime(\xi)) d \xi \cdot ( z - \frac{Q}{z} + \ep \phi^\prime(z)z) \right) \\
	& - \frac{A^4 \ep}{2 \pi^2} {\rm Im} \left( -i \int_{S^1} \frac{|\xi + \frac{Q}{\xi} + \ep \phi(\xi) |^2}{1 - A^2 \ep^2 (\xi + \frac{Q}{\xi} + \ep \phi(\xi)) (z + \frac{Q}{z} + \ep \phi(z)) } (1 - \frac{Q}{\xi^2} + \ep \phi^\prime(\xi)) d \xi \cdot ( z - \frac{Q}{z} + \ep \phi^\prime(z)z) \right),
\end{split}
\ee
in which $A$ is uniquely determined by $\ep \phi$ due to \eqref{E:A}. Firstly, we have the regularity concerned with the functional 
\begin{lemma} \label{L:Regularity}
	Let $s>\frac{3}{2}$. Suppose $\delta>0$ is a small number and  $\mathbf{B}_{X^s,\delta}$ is a neighbourhood of $0$ in $X^s$ with $\|f\|_{X^s} <\delta$. Then the nonlinear mapping $\mathcal{F} \in C^\infty\left(\mathbf{B}_{X^s,\delta} \times [0,\frac{1}{3}),Y^{s-1}\right)$.
\end{lemma}
\begin{proof}
Let 
\[
\psi(x,t) = \Psi(x,t) - \frac{\Omega}{2}|x|^2, \quad x \in D
\]	
be the relative stream function	for $\omega = \frac{1}{\pi \ep^2}\mathbf{I}_K$. Note that $\mathcal{F}(\phi,\ep)$ is the tangential derivative of $\psi$ on the boundary of vortical domain $\p K$, namely 
\[
\mathcal{F}(\phi,\ep) = \frac{1}{|\p_z \Phi(e^{i\theta})|} \p_\theta \psi(\phi,\ep)(\theta).
\]
It is equivalent to consider the regularity of nonlinear functional $\psi(\varphi)(e^{i \theta})$ with respect to boundary variation of $K$.  To this end, let us consider the free-boundary elliptic problem
\be
\begin{cases}
		 -\Delta \psi  = \frac{1}{\pi \ep^2} \mathbf{I}_K - \frac{1-Q^2}{2\pi \ep^2}, \\
		 \psi  = f \text{ on } \p K, \\
		 \psi  = - \frac{1-Q^2}{8 \pi \ep^2} \text{ on } \p D.
\end{cases}
\ee
Clearly $\psi$ is smooth on either $K$ or $D\setminus K$. The standard elliptic regularity theorem implies $\psi \in W^{2,p}(D)$ for  $1 \leq p<\infty$, and $\psi \in C^{1,\alpha}$ due to Sobolev embedding theorem. It implies that the normal derivatives of $\psi$ from both sides should be coincided, i.e.,
	\[
	\mathbf{n} \cdot \nabla \left(\mathcal{H}_K(f) + \Delta_K^{-1}(\frac{1+Q^2}{2\pi \ep^2} )\right) + \mathbf{n} \cdot \nabla \left(\mathcal{H}_{D \setminus K}(f) -\Delta_{D\setminus K}^{-1}(\frac{1-Q^2}{2\pi \ep^2} )\right)=0
	\]
where $\Delta_S$ denotes the Dirichlet Laplacian for the domain $S$. It implies
	\[
	f = \left(\mathcal{N}_K + \mathcal{N}_{D \setminus K} \right)^{-1} \left( \mathbf{n} \cdot \nabla (\Delta_{D\setminus K}^{-1}(\frac{1-Q^2}{2\pi \ep^2}) -\Delta_K^{-1}(\frac{1+Q^2}{2\pi \ep^2} )) \right)
	\]
	hence the regularity of $\mathcal{F}(\phi,\ep)$ follows form the $C^\infty$ regularity of the Dirichlet-Neumann operator, the Laplacian with respect to the boundary variation is $C^\infty$ smooth and the operator $\mathcal{N}_K + \mathcal{N}_{D \setminus K}$ is positively definite, cf. Appendix of \cite{Shatah2013}.
\end{proof}
The other essential ingredient is the linearization of $\mathcal{F}$ at the rescaled Kirchhoff ellipse patch.  
\begin{lemma} \label{L:Lin-Elli}
Suppose $h(\theta) = \sum_{n \geq 2} B_n \cos n \theta \in X^s$. We have
\be
\mathcal{D}_{\phi}\mathcal{F}(0,0) h(\theta) = \sum_{n \geq 1} y_{n+1} \sin n \theta,
\ee
where
\be\label{E:Spec}
\begin{split}
& y_2=-\frac{1}{2 \pi} (1+Q)^2 B_2,\;\; y_3= -\frac{2}{\pi}Q^2 B_3,\\
& y_{n+1} = \frac{1}{\pi} (\frac{1-Q^2}{2}n-1-Q^n)(B_{n+1}-QB_{n-1}).
\end{split}
\ee
Moreover, it is an isomorphism for $X^s$ to $Y^{s-1}$ for
\[
\frac{1-Q^2}{2}n-1-Q^n \neq 0, \quad \forall n \geq 3.
\]
\end{lemma}
\begin{proof}
Denote by $h(w) = \sum_{n \geq 2} \frac{B_n}{w^n}$. Due to \eqref{E:non-linear}, we let $\mathcal{F}(\ep \phi,\ep) = I_1 + I_2$ where
\[
\begin{split}
I_1(\ep \phi,\ep)&: = \frac{A^2(1-Q^2)}{2 \pi \ep}   {\rm Im} \left( \overline{z + \frac{Q}{z} + \ep \phi(z)} \cdot ( z - \frac{Q}{z} + \ep \phi^\prime(z)z) \right) \\
 & + \frac{A^2}{2 \pi^2 \ep} {\rm Im} \left(-i \int_{\mathbb{T}} \frac{ \overline{ \xi + \frac{Q}{\xi} - z - \frac{Q}{z}} + \ep (\overline{\phi(\xi)} - \overline{\phi(z)})}{ \xi + \frac{Q}{\xi} - z - \frac{Q}{z} + \ep (\phi(\xi) - \phi(z)) } (1 - \frac{Q}{\xi^2} + \ep \phi^\prime(\xi)) d \xi \cdot ( z - \frac{Q}{z} + \ep \phi^\prime(z)z) \right)
\end{split}
\]	
\[
\begin{split}
I_2(\ep \phi,\ep) :& =- \frac{A^4 \ep}{2 \pi^2} {\rm Im} \bigg( -i \int_{\mathbb{T}} \frac{|\xi + \frac{Q}{\xi} + \ep \phi(\xi) |^2}{1 - A^2 \ep^2 (\xi + \frac{Q}{\xi} + \ep \phi(\xi)) (z + \frac{Q}{z} + \ep \phi(z)) } (1 - \frac{Q}{\xi^2} + \ep \phi^\prime(\xi)) d \xi \\
& \cdot ( z - \frac{Q}{z} + \ep \phi^\prime(z)z) \bigg),
\end{split}
\]
in which $A(\ep \phi,\ep) = \frac{1-Q^2}{1-Q^2 - \frac{\ep^2}{2\pi} \|\phi\|^2_{\dot{H}^{\frac{1}{2}}(S^1)}}$. It is not hard to check that $\mathcal{D}_\phi A(0,0) h=0$ for any $h \in X^s$ hence it is sufficient to let $A=1$ in the rest of the proof.
On the other hand, a straightforward computation shows that  
\be
\mathcal{D}_\phi I_2(0,0)h(\theta) = 0, \quad \forall h \in X^s.
\ee
since the presense of $\ep$ in $I_2(\ep \phi,\ep)$ as linear term. Therefore it is enough to consider the linearization of $I_1(\phi,\ep)$ with respect to boundary variation. We have  
\[
\begin{split}
\mathcal{D}_\phi I_1(0,0) h(\theta) & = \frac{1-Q^2}{2 \pi} {\rm Im} \left(\overline{z+ \frac{Q}{z}} h(z) + \overline{h(z)}(z- \frac{Q}{z})\right) \\
& + \frac{1}{2 \pi} {\rm Im} \left(-i \int_{S^1} \frac{\overline{\phi(\xi)} - \overline{\phi(z)}}{ \xi + \frac{Q}{\xi} - z - \frac{Q}{z} } (1 - \frac{Q}{\xi^2} ) d\xi \cdot (z-\frac{Q}{z}) \right) \\
& + \frac{1}{2 \pi} {\rm Im} \left( i \int_{S^1} \frac{ \overline{ \xi + \frac{Q}{\xi} - z - \frac{Q}{z}} }{ (\xi + \frac{Q}{\xi} - z - \frac{Q}{z})^2} (1 - \frac{Q}{\xi^2}) (\overline{h(\xi)} - \overline{h(z)}) d \xi \cdot ( z - \frac{Q}{z}) \right) \\
& + \frac{1}{2 \pi} {\rm Im} \left( i \int_{S^1} \frac{ \overline{ \xi + \frac{Q}{\xi} - z - \frac{Q}{z}} }{ \xi + \frac{Q}{\xi} - z - \frac{Q}{z}} h^\prime(\xi) d \xi \cdot ( z - \frac{Q}{z}) \right) \\
&  + \frac{1}{2 \pi} {\rm Im} \left( i \int_{S^1} \frac{ \overline{ \xi + \frac{Q}{\xi} - z - \frac{Q}{z}} }{ \xi + \frac{Q}{\xi} - z - \frac{Q}{z}} (1-\frac{Q}{\xi^2}) d \xi \cdot h^\prime(z)z \right), \;\; \text{ where } z=e^{i \theta}
\end{split}
\]	
The formula of linearization indeed recovers from Proposition 2 of \cite{Hmidi2016} hence the spectral information \eqref{E:Spec} follows straightforwardly.  Moreover, let $f_{n+1} = \frac{\pi}{\frac{1-Q^2}{2}n-1-Q^2} y_{n+1}$. Due to \eqref{E:Spec} we have 
\[
B_{n+1} =Q^{n-2} B_3 + \sum_{0 \leq m \leq n-3} Q^{m} f_{n+1-m}.
\]
The boundedness of the inverse operator $\left(\mathcal{D}_\phi I_1(0,0)\right)^{-1} : Y^{s-1} \rightarrow X^s$ follows form
\[
\frac{1-Q^2}{2} n -1 -Q^n \neq 0, \quad  \forall n \geq 3.
\]
The proof is finsihed.
\end{proof}
Now we are going to prove Theorem \ref{T:main-1}. 
\begin{proof}[Proof of Theorem \ref{T:main-1}]
Note that one has $\mathcal{F}(0,0)=0$. Let $f(Q) = \frac{1-Q^2}{2}n-1-Q^n$. Clearly $Q=\frac{1}{2}, n=3$ is a solution of $f(Q)=0$. We note that for any $0<Q<\frac{1}{2}$, $f(Q) \neq 0$ for any $n \geq 3$. Indeed, one has
\[
\frac{1-Q^2}{2} n -1 -Q^n > \frac{3}{8} n - 1 -\left(\frac{1}{2}\right)^n > \frac{9}{8} -1 - \left(\frac{1}{2}\right)^3=0
\]
 Due to Lemma \ref{L:Lin-Elli}, we have $\mathcal{D}_\phi \mathcal{F}(0,0):X^s\rightarrow Y^{s-1}$ is an isomorphism. Applying the Implicit Function theorem on $\mathcal{F}(\ep\phi,\ep) =0$ it follows that there exists $\ep_0>0$ such that for any $0<\ep <\ep_0$, one has
\[
\mathcal{F}(\ep \phi(\ep),\ep) = 0.
\]
It implies that $\p( \ep^{-1} K)$ is an $O(\ep^2)$ perturbation of the scaled ellipse $\mathbf{E}_Q$. The proof is completed.
\end{proof}

\section{Uniformly rotating vorticity with multiple vortical domains} \label{S:Multi} 

Upon the existence of rapidly rotating simply-connected vortex patch, we continute to construct rapidly rotating vortex patch with multiple vortical conponents in this section. Suppose $0<Q<\frac{1}{2}$. As mentioned in Secton \ref{S:Intro}, we consider the vorticity field
\be \label{E:Multi-vor-1}
\omega = \frac{1}{\pi r_0^2} \mathbf{I}_{D_0} + \sum_{j=1}^2 \omega_j,
\ee
where $|D_0|= \pi r_0^2(1-Q^2)$ and the vortical component $\omega_j$ satisfies 
\begin{gather} \label{E:Multi-con}
	{\rm supp \;} \omega_j = D_j,\; \mu:=\int_{D} \omega_j>0, \; |D_j|= \pi  r_j^2, \; {\rm dist } (D_j) \leq 2.5 r_j, \; j=1,2, 
\end{gather}
in which $\vec{r}=(r_0,r_1,r_2) \in \BFR_+^{3}$ are small parameters. In particular we suppose $\max\{r_1,r_2\} <\frac{r_0^2}{2}$ therefore $|\vec{r}| :=\max\{r_0,r_1,r_2\}=r_0$. 

The stream function is given by 
\be 
\Psi(x) =\frac{1}{\pi r_0^2} \int_{D_0} G(x,y)  d \mu_y  + \sum_{j=1}^{2} \int_{D_j} G(x,y) \omega_j(y) d \mu_y.
\ee
hence the velocity follows 
\be
	u(x)  =\nabla^\perp \Psi(x) = u_o(x) + u_e(x),
\ee
in which
\[
	 u_o(x): = \frac{1}{\pi r_0^2} \int_{D_0} \nabla_x^\perp G(x,y)  d \mu(y), \quad 
	 u_e(x):=\sum_{j=1}^{2} \int_{D_j} \nabla_x^\perp G(x,y) \omega_j(y) d \mu(y).
\]
Due to  \eqref{E:Contour-F}, $u_0$ is furnished in a contour integral formula 
\be
\overline{u_o(z)} = \frac{1}{4 \pi^2 r_0^2} \int_{\p D_0} \frac{\overline{\xi} - \overline{z}}{\xi -z} d \xi - \frac{1}{4 \pi^2 r_0^2} \int_{\p D_0} \frac{|\xi|^2}{1-\overline{z} \xi } d \xi.
\ee
By Lemma \ref{L:Steady-vor}, $\omega$ is a uniformly rotating vorticity if the following conditions hold:
\begin{enumerate}
\item The relative stream function $\psi= \Psi -\frac{\Omega}{2}|x|^2$ solves $-\Delta \psi = f_j(\psi)$ in each vortical domain $D_j$ for some vorticity profiles $f_j \in C^1$, $j=1,2$. 

\item The vortex patch $\omega_0 = \frac{1}{\pi r_0^2} \mathbf{I}_{D_0}$ solves 
\be \label{E:Central-patch}
\begin{split}
& {\rm Im} \left( \left(2 \Omega \overline{z} + I(z) + 2\overline{u_e} \right) \mathbf{n}_0 \right) = 0, \quad z \in \p D_0, \\
& I(z) =  \frac{1}{2 \pi^2 r_0^2} \int_{\p D_0} \frac{\overline{\xi} - \overline{z}}{\xi -z} d \xi - \frac{1}{2 \pi^2 r_0^2} \int_{\p D_0} \frac{|\xi|^2}{1-\overline{z} \xi } d \xi,
\end{split}
\ee 
where we denote the unit outward normal to $\p D_0$ by $\mathbf{n}_0$. 
\end{enumerate} 
In the rest of this section, we shall mainly focus on the case of vortex patch, namely $\omega_j = \frac{\kappa}{\pi r_j^2} \mathbf{I}_{D_j}, j=1,2$ and sketch the proof for $C^{0,1}$ concentrated vorticity. It is reasonable since the essential aspect in the both analysis is indeed the location of the vorticity components.

 \subsection{Vortex patches} \label{SS:VP} 

Suppose that $\omega_j = \frac{\mu}{\pi r_j^2} \mathbf{I}_{D_j}, j=1,2$ are piecewisely constant vortex patches with prescribed total vorticity $\mu >0$. Since the vortical domains are the main unknowns, as mentioned in Section \ref{S:Pre}, we parameterize them by conformal mappings. 

As Section \ref{S:Sin-patch}, we parameterize the exterior domain of $D_0$ by conformal mapping $\Gamma_0$ hence the boundary $\p D_0$ is given as
\be \label{E:Phi-1}
\Gamma_0(z) = Ar_0\left( z+ \frac{Q}{z} + \sum_{n \geq 2} \frac{B_n}{z^n} \right),\quad z: =e^{i \theta} \in S^1,
\ee  
which is uniquely determined by 
\be
\beta_0(\theta) =  \sum_{n \geq 2} B_n \cos n \theta \in X^s
\ee
and $A$ depends on $\beta_0$ as \eqref{E:A}. For the vortical domains $D_j, j=1,2$, on the other hand, by letting $k=2$ in Lemma \ref{L:CM} one has the vortical domain $D_j$ is uniquely parameterized by the conformal mapping 
\[
\Gamma_j(z) = x_j + a_1^j r_j\left( z+ \wt \Gamma_j(z)\right), \quad j=1,2,
\] 
in which $\wt \Gamma_j$ describes the shape of domain $D_j$ and $x_j$ specifies the location. Moreover, due to the $S^1$-invariance, we have $\wt \Gamma_j(z) = \sum_{n \geq 3} A_n^j z^n, \; A_n^j \in \BFR$ which is uniquely determined by
\be
\beta_j(\theta) \in Z_e^s:=\left\{ f (\theta) = \sum_{ n \geq 3} A_n \cos n \theta \in \dot{H}^s(S^1) \right\}, 
\ee  
and suppose
\be
x_1 = \left(1 - \frac{2\mu \pi}{1-Q^2} r_0^2(1+ r_0^2 y_1) \right), \quad x_2 = -\left(1 -  \frac{2\mu \pi}{1-Q^2}r_0^2(1+ r_0^2 y_2) \right).
\ee 
since they would be rather close to the boundary. Suppose 
\[
W_o^s:=\left\{ f(\theta) = \sum_{ n \geq 2} A_n \sin n \theta \in H^s(S^1) \right\}.
\]
Let $\vec{\beta} := (\beta_0,\beta_1,\beta_2)$, $Y:=(y_1, y_2)$, $\vec{r}=(r_0,r_1,r_2)$, $M_{\vec{r}} \beta= (r_0\beta_0,r_1\beta_1,r_2\beta_2)$. Motiviated by Lemma \ref{L:Steady-vor} and \eqref{E:Central-patch}, we consider the nonlinear functional 
\[
\mathcal{F}\left( \vec{\beta},Y,\vec{r} \right)(\theta) := \left(\mathcal{F}_0(\theta), \mathcal{F}_1(\theta),\mathcal{F}_2(\theta) \right):  X^s \times (Z_e^s)^2 \times \BFR_+^2 \times \BFR_+^{3} \rightarrow Y^{s-1} \times (W_o^{s-1})^2,
\]
where 
\be
\begin{split}
	& \mathcal{F}_0(M_{\vec{r}}\beta,Y,\vec{r})(\theta)  := {\rm Im} \left( \left(\Omega \overline{\Gamma_0(z)} + \frac{1}{2} I(\Gamma_0(z)) \right) \frac{z\p_z \Gamma_0(z)}{|\p_z \Gamma_0(z)|} \right) \\
	& + \frac{1}{|\p_z \Gamma_0(z)|} \p_\theta \left( \sum_{j=1}^{2} \int_D \frac{\mu}{\pi r_j^2} G\left(\Gamma_0(z),\Gamma_j(y)\right) |\p_z \Gamma_j(y)|^2 d \mu(y) \right), \quad z=e^{i \theta} \in S^1,  
\end{split}
\ee
and 
\be 
\begin{split}
	& \mathcal{F}_j(M_{\vec{r}}  \beta,Y,\vec{r})(\theta): = \\
	&  -{\rm Im}\left( \frac{1}{4\pi^2 r_0^2 i} \frac{z\p_z \Gamma_j(z)}{|\p_z \Gamma_j(z)|} \left( \int_{S^1} \frac{\overline{\Gamma_0(\xi) -\Gamma_j(z)}}{\Gamma_0(\xi) -\Gamma_j(z)} \p_z \Gamma_0(\xi) d \xi - \int_{S^1} \frac{|\Gamma_0(\xi)|^2}{1-\overline{\Gamma_j(z)} \Gamma_0(\xi)} \p_z \Gamma_0(\xi) d \xi \right) \right) \\
	&  + \frac{1}{|\p_z \Gamma_j(z)|}  \p_\theta \left( \sum_{k=1}^{2} \frac{\mu}{\pi r_k^2} \int_{D} G(\Gamma_j(z),\Gamma_k(y)) |\p_z \Gamma_k(y)|^2 d \mu_y - \frac{\Omega}{2} |\Gamma_j(z)|^2 \right), \quad z = e^{i \theta} \in S^1, j=1,2.
\end{split}
\ee
Similar with Lemma \ref{L:Regularity}, we have 
\begin{lemma} 
There exists small number $\delta>0$ such that the nonlinear functional $\mathcal{F}(\vec{\beta},X,;\vec{r}): B_{X^s \times (Z_e^s)^2,\delta} \times \BFR_+^2 \times [0,\frac{1}{3})^{3} \rightarrow Y^{s-1} \times (W_o^{s-1})^2$ is a $C^\infty$ mapping.
\end{lemma}   
Denote by
\begin{gather*}
h_0(r_0 \beta_0) :={\rm Im} \left( \left(\Omega \overline{\Gamma_0(z)} + \frac{1}{2} I(\Gamma_0(z)) \right) \frac{z\p_z \Gamma_0(z)}{|\p_z \Gamma_0(z)|} \right), \\
H_0(\vec{\beta},Y,\vec{r})=R_0(M_{\vec{r}}\beta,Y,\vec{r}): =\frac{1}{|\p_z \Gamma_0(z)|} \p_\theta \left( \sum_{j=1}^{2} \frac{\mu}{\pi r_j^2}\int_D G(\Gamma_0(z),\Gamma_j(y)) |\p_z \Gamma_j(y)|^2 d \mu(y) \right), \\
h(r \beta)(\theta):= -\frac{\mu}{2\pi^2} \p_\theta \int_{D} \log \left| \wt \Gamma(z) - \wt \Gamma(y) \right| |1 + \p_z \wt \Gamma(z)|^2 d \mu_z, \\
H_j(\vec{\beta},Y,\vec{r})=R_j(M_{\vec{r}}\beta,Y,\vec{r}) : = -{\rm Im}\bigg( \frac{1}{4\pi^2 r_0^2 i} \frac{z\p_z \Gamma_j(z)}{|\p_z \Gamma_j(z)|} \big( \int_{S^1} \frac{\overline{\Gamma_0(\xi) -\Gamma_j(z)}}{\Gamma_0(\xi) -\Gamma_j(z)} \p_z \Gamma_0(\xi) d \xi \\
- \int_{S^1} \frac{|\Gamma_0(\xi)|^2}{1-\overline{\Gamma_j(z)} \Gamma_0(\xi)} \p_z \Gamma_0(\xi) d \xi \big) \bigg) \\
+  \frac{1}{|\p_z \Gamma_j(z)|}  \frac{\mu}{\pi r_j^2} \p_\theta \left( \int_{D}  g(\Gamma_j(z),\Gamma_j(y)) |\p_z \Gamma_j(y)|^2 d \mu_y - \frac{\Omega}{2} |\Gamma_j(z)|^2 \right) \\
+  \frac{1}{|\p_z \Gamma_j(z)|} \frac{\mu}{\pi r_j^2}  \p_\theta \left( \int_{D} G(\Gamma_j(z),\Gamma_k(y)) |\p_z \Gamma_k(y)|^2 d \mu_y \right),\; j=1,2, k \neq j.
\end{gather*}
By Lemma \ref{L:Lin-Elli}, one has $\mathcal{D}h_0(0)$ is an isomorphism from $X^s$ to $Y^{s-1}$ provided $0<Q<\frac{1}{2}$. On the other hand, since the variation of $|1+ \p_z \wt \Gamma_j(z)|$ depends on $r_j$ linearly, due to Lemma 3.1 of \cite{LWZ2019}, one has 
\begin{lemma} \label{L:L-vp-1}
It holds that $h \in C^\infty\left(B_{Z_e^s,\delta}, W_o^{s-1}\right)$ for any $s>\frac{3}{2}$, $h(0)=0$ and 
\be
\mathcal{D}h(0) \left( \sum_{n \geq 3} A_n \cos n \theta\right) = -\sum_{n \geq 2} \frac{n-1}{2 \pi} A_{n+1} \sin n\theta.
\ee
\end{lemma}

On the other hand, we have 
\begin{lemma} \label{L:Point}
$H_j(\vec{\beta}, Y,\vec{r}) \in C^\infty\left(B_{X^s \times (Z_e^s)^2,\delta} \times \BFR_+^2 \times [0,\frac{1}{3})^{3}, Y^{s-1} \times (W_o^{s-1})^2 \right),\;j=0,1,2$ and 
\[
H_0(\vec{\beta},Y,0):=\lim_{|\vec{r}| \rightarrow 0} R_0(M_{\vec{r}}\beta, Y,\vec{r}) =0, \quad 
H_j(\vec{\beta},Y,0):=\lim_{|\vec{r}| \rightarrow 0} R_j(M_{\vec{r}}\beta, Y,\vec{r}) = \frac{1-Q^2}{4 \pi^2 \mu} y_j \sin \theta, \;\; j=1,2.
\]
Moreover, denote by $\vec{H} = (H_1,H_2)^t$, one has $\mathcal{D}_Y \vec{H}(0,0,0)$ is invertible. 
\end{lemma}
\begin{proof} 
Note that $x_j = \pm (1 - \frac{2  \mu}{1-Q^2} r_0^2(1 + r_0^2 y_j))$. We have $x_j \rightarrow (\pm 1,0)$ as $\vec{r} \rightarrow 0$ which implies $\nabla_1 G(x_j,x_k) = \nabla_1 G(x_j,0)= 0$ for $j\neq k$. One has 
\[
H_0(\vec{\beta},Y,\vec{r}) = \frac{iz\p_z \Gamma_0(z)}{|\p_z \Gamma_0(z)|} \cdot \bigg( \sum_{j=1}^2 \frac{\mu}{\pi r_j^2} \int_D \nabla_1 G(\Gamma_0(z),x_j + r_ja_1^j(z + \wt \Gamma_j(y))) r_j^2 (a_1^j)^2 r_j^2 |1 + \p_z \wt \Gamma_j(y)|^2 d \mu(y) \bigg).
\]
It follows
\[
H_0(\vec{\beta},Y,0) = ie^{i \theta} \cdot \left(\sum_{j=1}^2 \mu \nabla_1 G(0,x_j) \right) = 0.
\]
On the other hand,

note that $\Omega = \frac{1-Q^2}{4\pi r_0^2} + \frac{\mu}{4 \pi}$, one has   
\[
\begin{split}
H_j(0,Y,0) & = \lim_{|\vec{r}| \to 0} \left( \mu \nabla_1 g(x_j,x_j) - \Omega x_j \right) \cdot i e^{i \theta} + \left( \mu \nabla_1 G(x_j,x_k) + (1-Q^2) \nabla_1 G(x_j,0) \right) \cdot i e^{i \theta} \\
& = \lim_{|\vec{r}| \to 0} ( \frac{\mu}{\pi(1-|x|^2)} - \Omega) x_j \cdot i e^{i \theta} \\
& = \lim_{|\vec{r}| \to 0} \left( \frac{\mu}{2 \pi} \frac{1}{\frac{2  \mu}{1-Q^2} r_0^2} (1-r_0^2y_j)(1+ \frac{2  \mu}{1-Q^2} \frac{r_0^2}{2}) - \Omega + O(r_0^2) \right) \cdot i e^{i \theta} \\
& = \lim_{|\vec{r}| \to 0}  \left( \frac{1-Q^2}{4\pi} \frac{1}{r_0^2} (1-r_0^2y_j + \frac{2  \mu}{1-Q^2} \frac{r_0^2}{2}) - \Omega + O(r_0^2) \right) \cdot i e^{i \theta} \\
& = \frac{1-Q^2}{4\pi} y_j \sin \theta.
\end{split}
\]	
Therefore, we have 
\[
\mathcal{D}_{Y} \vec{H}(0,0,0) = {\rm diag} \left(\frac{1-Q^2}{4 \pi^2} \sin \theta,  \frac{1-Q^2}{4 \pi^2} \sin \theta \right)
\]

\end{proof}

We are in the position to prove Theorem \ref{T:main-2} for the patch case. 
\begin{proof}
Due to Lemma \ref{L:Lin-Elli}, Lemma \ref{L:L-vp-1} and Lemma \ref{L:Point}, we have 
\[
\mathcal{D}_{\vec{\beta},Y} \mathcal{F}(0,0,0) \in L\left(X^s \times (Z_e^s)^2 \times \BFR_+^2, Y^{s-1} \times (W_o^{s-1})^2 \right)
\]
is an isomorphism and $\mathcal{F}(0,0,0)=0$. The implicit function theorem implies that there exists $\ep_0 >0$ such that $\mathcal{F}(\vec{\beta}(r),Y(r),r) =0$ for $|\vec{r}|<\ep_0$. 
\end{proof}

\subsection{Concentrated $C^{0,1}$ vorticity} \label{SS:Lip}

Compared with vortex patch case, the main difference in the construction of concentrated $C^{0,1}$ rotating vorticity is on the choice of vorticty profle. Since $D_j$ is close to a disk, we start with the semilinear elliptic equation on $D_j$. For each $j=1, 2$, we take $f_j$ such that 
\begin{subequations}  \label{localpatch:gammaassumptions} 
	\be 
	f_j \in C^\infty (\BFR, \BFR), \; \text{odd},
	\qquad  f_j^\prime(0) < 0, \quad f_j (\tau) > 0 \textrm{ on } \BFR^-, 
	\ee 
	\be 
		\exists \; \text{a negative radial solution } \psi_j^* \text{ to } \quad \Delta \psi_j = f_j (\psi_j)  \; \textrm{ in } \; D, \qquad \psi_j|_{S^1} =0, 
	\label{E:f-2} \ee
	and 
	\be 
	\Delta - f_j^\prime(\psi_j^*) \textrm{ is non-degenerate}. \label{localpatch:nondegen}
	\ee \end{subequations}
More specifically, \eqref{localpatch:nondegen} requires that $\Delta - f_j^\prime(\psi_j^*)$, viewed as a self-adjoint operator on $L^2(D)$, does not contain $0$ in its spectrum; this is of course a generic condition.  Indeed, the class of $f_j$ satisfying \eqref{localpatch:gammaassumptions} is quite large (cf., e.g., \cite{brezis1997elliptic,figueiredo1982apriori,lions1982positive}). \\

As similar as Section 5 of \cite{LWZ2019}, our procedure to obtain Lipschitz concentrated steady vorticity is roughly the following:
\begin{itemize}

	\item  [{\bf Step 1.}] Through the conformal mapping coordinates $\Gamma_j$, consider the elliptic problem  
	\be \label{E:semiLE-1}
	\Delta \wt \psi = |1+ \p_z \wt \Gamma_j|^{2} f_j (\wt \psi), \qquad  \wt \psi: D \to \BFR, \quad \wt \psi|_{S^1}=0. 
	\ee
	By Implicit Function Theorem, there exists a unique solution $\wt \psi_j$ close to $\psi_j^*$ in $H^{s+\frac 32} (D)$ topology. Let 
	\be \label{E:Vor-j-1}
	\psi_j= \frac { \mu \wt \psi_j  \circ \Gamma_j^{-1} \mathbf{I}_{\Gamma_j(D)} } {\int_{D} 
		|1+ \p_z \wt \Gamma_j|^{2}f_j(\wt \psi_j) d\mu}  , \quad \omega_j = \frac {\mu \mathbf{I}_{\Gamma_j(D)}f_j (\wt \psi_j  \circ \Gamma_j^{-1})}{ (a_1^jr_j)^2  \int_{D} |1+ \p_z \wt \Gamma_j|^{2} f_j(\wt \psi_j) d\mu}.
	\ee
	Clearly 
	\be \label{E:Vor-j-2}
	\omega_j =\Delta \psi_j \; \text{ on }\; D_j = \Gamma_j(D) =supp(\omega_j), \quad \int_{D_j} \omega_j d\mu=\mu.
	\ee
	\item  [{\bf Step 2.}] Define 
	\be \label{E:Vor-1} 
	\omega=  \frac{1}{\pi r_0^2} \mathbf{I}_{D_0}+ \sum_{j=1}^2 \omega_j, \text{ and } \Psi = \Delta^{-1} \omega
	\ee 
	It holds that $\Delta \psi = \Delta \psi_j$ on $D_j$, $j=1, 2$. For $|\vec{r}| << \delta$, the domains $D_j$ are mutually disjoint. So if $\psi = const$ on each $\p D_j$, then $\psi-\psi_j=const$ on each $D_j$ and thus by Lemma \ref{L:Steady-vor},  $\omega$ is a steady solution to the Euler equation. Therefore consider 
	\be \label{E:phi-1}
	\mathcal{F}_j (M_{\vec{r}}\beta, Y, \vec{r}) (\theta) = \frac{1}{|\p_z \Gamma_j(z)|}\p_\theta \Psi\big( \Gamma_j(e^{i\theta}) \big) 
	\ee
	\[
	\begin{split}
		& \mathcal{F}_0(M_{\vec{r}}\beta,Y,\vec{r})(\theta)  := {\rm Im} \left( \left(\Omega \overline{\Gamma_0(z)} + \frac{1}{2} I(\Gamma_0(z)) \right) \frac{z\p_z \Gamma_0(z)}{|\p_z \Gamma_0(z)} \right) \\
		& + \frac{1}{|\p_z \Gamma_0(z)|} \p_\theta \left( \sum_{j=1}^{2} \int_D G\left(\Gamma_0(z),\Gamma_j(y)\right) (\omega_j\circ \Gamma_j)(y)|\p_z \Gamma_j(y)|^2 d \mu(y) \right), \quad z=e^{i \theta} \in S^1.
	\end{split}
	\]
	Equivalently, we have 
	
	\begin{lemma} \label{L:steady-C0-1}
		If $\mathcal{F} (\vec{\beta}, Y, \vec{r}):= \left(\mathcal{F}_0(M_{\vec{r}}\beta,Y,\vec{r}),\mathcal{F}_1(M_{\vec{r}}\beta,Y,\vec{r}), \mathcal{F}_2(M_{\vec{r}}\beta,Y,\vec{r})\right) \equiv 0$,  then the corresponding $\omega$ is a unfiormly rotating vorticity.
	\end{lemma}
\end{itemize}
The above condition will be achieved by choosing proper $Y$ and $\vec{\beta} \in (X^s) \times (Z_e^s)^2$ through a perturbation argument similar with subsection \ref{SS:VP}. Here we only need to point out the difference on the linearization 
\[
h_j(r_j \beta_j) = \frac{1}{|\p_z \Gamma_j(z)|} \p_\theta \left( \int_{D} \log|\wt \Gamma_j(z) - \wt \Gamma_j(y)| (\omega_j \circ \Gamma_j)(y) |\p_z \Gamma_j(y)|^2 d \mu(y) \right).
\]
\begin{lemma}
	There exists $\delta>0$ such that 
	\be
	\mathcal{D}h_j(0) \left(\sum_{n \geq 3} A_n \cos n \theta \right) = \sum_{n \geq 2} n \lambda_{j,n} A_{n+1} \sin n \theta 
	\ee 
with $\lambda_{j,n} \in \BFR$, $|\lambda_{j,n}| \geq \delta$.
\end{lemma} 
On the other hand, suppose
\[
\begin{split}
& H_0(\vec{\beta},Y,\vec{r}) =R_0(M_{\vec{r}}\beta,Y,\vec{r}): =\frac{1}{|\p_z \Gamma_0(z)|} \p_\theta \left( \sum_{j=1}^{2} \int_D G(\Gamma_0(z),\Gamma_j(y)) (\omega_j \circ \Gamma_j)(y) |\p_z \Gamma_j(y)|^2 d \mu(y) \right), \\
& H_j(\vec{\beta},Y,\vec{r})=R_j(M_{\vec{r}}\beta,Y,\vec{r}) \\
&	: = -{\rm Im}\left( \frac{1}{4\pi^2 r_0^2 i} \frac{z\p_z \Gamma_j(z)}{|\p_z \Gamma_j(z)|} \left( \int_{S^1} \frac{\overline{\Gamma_0(\xi) -\Gamma_j(z)}}{\Gamma_0(\xi) -\Gamma_j(z)} \p_z \Gamma_0(\xi) d \xi - \int_{S^1} \frac{|\Gamma_0(\xi)|^2}{1-\overline{\Gamma_j(z)} \Gamma_0(\xi)} \p_z \Gamma_0(\xi) d \xi \right) \right) \\
&	+  \frac{1}{|\p_z \Gamma_j(z)|}  \p_\theta \left( \int_{D} g(\Gamma_j(z),\Gamma_j(y)) (\omega_j \circ \Gamma_j)(y) |\p_z \Gamma_j(y)|^2 d \mu_y - \frac{\Omega}{2} |\Gamma_j(z)|^2 \right) \\
&	+  \frac{1}{|\p_z \Gamma_j(z)|}  \p_\theta \left( \int_{D} G(\Gamma_j(z),\Gamma_k(y))(\omega_j \circ \Gamma_j)(y)  |\p_z \Gamma_k(y)|^2 d \mu_y \right),\; j=1,2, k \neq j.
\end{split}
\]
We have 
\begin{lemma} 
	$R_j(M_{\vec{r}}\beta, Y,\vec{r}) \in C^\infty\left(B_{X^s \times (Z_e^s)^2,\delta} \times \BFR_+^2 \times [0,\frac{1}{3})^{3}, Y^{s-1} \times (W_o^{s-1})^2 \right),\;j=0,1,2$ and 
	\[
	H_0(\vec{\beta},Y,0):=\lim_{|\vec{r}| \rightarrow 0} R_j(M_{\vec{r}}\beta, Y,\vec{r}) =0, \quad 
	H_j(\vec{\beta},Y,0):=\lim_{|\vec{r}| \rightarrow 0} R_j(M_{\vec{r}}\beta, Y,\vec{r}) = \frac{1-Q^2}{4 \pi^2 \mu} y_j \sin \theta, \;\; j=1,2.
	\]
	Moreover, denote by $\vec{H} = (H_1,H_2)^t$, one has $\mathcal{D}_Y \vec{H}(0,0,0)$ is invertible. 
\end{lemma}
Then the proof of Theorem \ref{T:main-2} for $C^{0,1}$ continuous vorticity follows the implicit function theorem.


	

\bibliographystyle{plain}
\bibliography{VP}

\end{document}